\documentclass[12pt]{article}
\usepackage{mathrsfs}
\usepackage{amsfonts}
\usepackage{amsthm,amsmath,amssymb,anysize,longtable}
\usepackage{tikz}
\usepackage{float}
\usepackage{hyperref}
\usepackage{xcolor}
\usepackage{textcomp}
\usepackage{makecell}
\usepackage{longtable}
\usepackage{array}
\usepackage{caption}
\usepackage{float}
\usepackage{booktabs}
\usepackage{enumitem}

\newtheorem{theorem}{Theorem}
\newtheorem{lemma}[theorem]{Lemma}
\newtheorem{proposition}[theorem]{Proposition}

\newtheorem{example}[theorem]{Example}
\newtheorem{definition}[theorem]{Definition}
\newtheorem{corollary}[theorem]{Corollary}
\usepackage{xcolor}

\usepackage[numbers,sort&compress]{natbib}
\usepackage{amsmath}
\allowdisplaybreaks[3]
\marginsize{4.2cm}{4.2cm}{3.6cm}{3cm}
\numberwithin{equation}{section}

\author{\textsc{Baojin Zhang and Liming Tang\footnote{corresponding author}}\\
\small{School of Mathematical Sciences}\\
\small{Harbin Normal University}\\
\small{150025 Harbin, China}\\
\small{E-mail: limingtang@hrbnu.edu.cn}}

\date{ }
\date{ }
\usepackage[symbol]{footmisc}

\begin{document}

\thispagestyle{empty}

\noindent{\Large
Nilpotentizers and the Nilpotent Graphs: Structural Insights into Lie Superalgebras}
 \footnote{
The work is supported by  the NSF of Hei Longjiang Province (No. LH2024A014)
}

	\bigskip
	
	 \bigskip

\begin{center}	
	{\bf
		
    Baojin Zhang\footnote{School of Mathematical Sciences, Harbin Normal University, 150025 Harbin, China; \ 2024400038@stu.hrbnu.edu.cn}
   and
Liming Tang\footnote{School of Mathematical Sciences, Harbin Normal University, 150025 Harbin, China; \ limingtang@hrbnu.edu.cn}\footnote{corresponding author}}
\end{center}

 \begin{quotation}
{\small\noindent \textbf{Abstract}:
In this paper, we systematically investigate the nilpotentizer and nilpotent graph for a Lie superalgebra over the field of characteristic $p\neq 2$.
First, we establish some fundamental properties of the nilpotentizer.
Next, we show that the nilpotent graph is one of the isomorphic invariants of Lie superalgebras.
Furthermore, we introduce the nilpotency measure 
which provides a quantitative assessment of nilpotency for a Lie superalgebra.
Finally, we use category theory to establish connections between Lie superalgebras and their nilpotent substructures, based on the construction of the nilpotentizer.

\medskip
 \vspace{0.05cm} \noindent{\textbf{Keywords}}:
Lie superalgebras; nilpotentizers; graph theory.

\medskip

\vspace{0.05cm} \noindent \textbf{Mathematics Subject Classification
2020}: 17B05, 17B30, 05C90, 18A05}
\end{quotation}
 \medskip

\section*{Introduction}
\ \ \
Nilpotency is a fundamental concept in the study of algebraic structures and representation theory.
Numerous researchers have focused on characterizing global nilpotency \cite{GMY1,CMZ,Salamon,Blackburn,Camina}.
In contrast, the detailed research of local properties remain to be explored.
The nilpotentizer in groups and Lie algebras are introduced by \cite{AZ} and \cite{TDIL}, respectively.
For a fixed element in a group or an algebra, the nilpotentizer of the element consists of all elements that generate a nilpotent substructure together with it.
The nilpotentizer for a group or an algebra is the intersection of the nilpotentizers of all  elements in it.

Since the 1980s, the intersection theory of algebra and graph theory has offered profound insights
into characterizing the structural properties of finite groups. In 1981, J. S. Williams \cite{JSW} introduced the prime graph of a finite group and
provided a classification of groups with disconnected prime graphs and established key properties of the connected components of prime graphs.
Subsequently, in 2002, W. M. Kantor and \'{A}. Seress \cite{KWMS} proposed the prime power graph for a finite simple group of Lie type of characteristic $p$ and demonstrated that the most finite simple group of Lie type can be uniquely identified by its prime power graph.
In 2006, A. Abdollahi, S. Akbari, and H. R. Maimani \cite{AAM} introduced the non-commuting graph of groups
and established the connections between the graph properties and the group commutativity.
In 2010, A. Abdollahi and M. Zarrin \cite{AZ} introduced the non-nilpotent graph of groups
and established a direct link between the non-nilpotent graphs and the hypercenter of groups.
Further extensions include the introduction of the solvableizer and the solvable graph by P. Bhowal, D. Nongsiang, and R. Nath in 2020 \cite{BN} which provided a framework for characterizing the solvability of groups through graph theory.
In 2025, D. Towers, I. Guti\'{e}rrez, and I. Fern\'{a}ndez \cite{TDIL} introduced nilpotentizer and nilpotent graph for a Lie algebra
and developed computational frameworks based on GAP and SageMath. In 2025, A.Y. Hummdi, K. Alnefaie, G. Scudo and S. A. Mohammad in \cite{HAS} extended the definitions to nilpotentizer and nilpotent graph to a Lie superalgebra $L$ and investigated some properties  the nilpotent graphs for
the triangular Lie superalgebra $\operatorname{t}(m|n,\mathbb{F}_{q})$ and the special linear Lie superalgebra $\frak{sl}(1|1, \mathbb{F}_{q})$.
 The additional related researches are in the papers \cite{DG,GCM,Zr,DN,NS,BC,KVTC,GKLNS,THu,cmmo}. All these results offer the powerful tools to investigate  the structure of algebras or groups.

 The aim in this paper is to establish the general results about nilpotizier and nilpotent graph for a Lie superalgebra over the field of characteristic $p\neq 2$, which provide effective tools for analyzing such algebraic structures.
For the case of $p=2$, the structure of the Lie superalgebras differs significantly, see \cite{BGLAEA,DLeites,AKALDLIS}.
Main results in this paper are as follows:

\begin{itemize}
\item For a Lie superalgebra $L$ over the complex number field, we have $$D(\operatorname{nil}(L))\subseteq E(L), \text{ for all } D\in \operatorname{Der}(L)$$ where $E(L)$ denotes the set of all Engel elements of $L$ and  $\operatorname{Der}(L)$ denotes the derivation algebra of $L$.


\item For a Lie superalgebra $L$, the nilpotency measure $\mu(L)$ is introduced, which quantifies how close the Lie superalgebra is to being nilpotent.
In particular, if $\mu(L)\neq 0$, then $L$ is not nilpotent.

\item Given any $q \geq 2$, for any $\varepsilon > 0$, there exists a Lie superalgebra $L$ such that
$$\left| \mu(L) - \frac{q}{q+1} \right| < \varepsilon.$$


\item From the view of category theory, a functor $\mathcal{N}$ is introduced, which links Lie superalgebras to their nilpotent substructures.

\item The revisions of Theorem 3.2(2) and Corollary 7.2 in \cite{HAS} are proposed. For details, see Remark (1) in Section \ref{section3} and Theorem \ref{theorem5} in this paper.
\end{itemize}

It is important to note that various results in this paper depend on the characteristic $\operatorname{char} \mathbb{F}$ of the ground field $\mathbb{F}$.
In this paper, we try to elucidate the local properties of a graded algebra and reveal some connections between the local properties and the global properties in supersymmetric settings.
\section{Basic}\label{section2}\

In this section, we revisit the key concepts and properties of  a Lie superalgebra over the field of characteristic $p\neq 2$.
For more concepts, readers may refer to \cite{KVG}.

Throughout this paper, $\mathbb{F}$ denotes a field;
$\operatorname{char} \mathbb{F}$ denotes the characteristic of $\mathbb{F}$; $\mathbb{Z}$ denotes the set of integers;
and $\mathbb{N}$ denotes the set of positive integers;
$\mathbb{Z}_{2} := \mathbb{Z}/2\mathbb{Z}$ denotes the ring of integers modulo 2.
A superalgebra over $\mathbb{F}$ is a $\mathbb{Z}_{2}$-graded $\mathbb{F}$-algebra $A = A_{\bar{0}} \oplus A_{\bar{1}}$
(that is ,if $a\in A_{\alpha}$, $b\in A_{\beta}$, $\alpha,\beta\in \mathbb{Z}_{2}$, then $ab\in A_{\alpha+\beta}$).
A Lie superalgebra is a superalgebra $L = L_{\bar{0}} \oplus L_{\bar{1}}$
with an operation $[\cdot, \cdot]$ satisfying the following axioms:
\begin{enumerate}[label=(\arabic*)]
\item $[x,y]=-(-1)^{|x||y|}[y,x]$, \text{ for } $x,y\in L_{\bar{0}}\cup L_{\bar{1}}$,
\item $[x,[y,z]]=[[x,y],z]+(-1)^{|x||y|}[y,[x,z]]$, \text{ for } $x,y,z\in L_{\bar{0}}\cup L_{\bar{1}}$.
\end{enumerate}
Here $|x|=0$ if $x\in L_{\overline{0}}$ and $|x|=1$ if $x\in L_{\overline{1}}$. If $\operatorname{char} \mathbb{F}=3$,
The Jacobi identity for Lie superalgebras entails, additionally, that
$$[x,[x,x]]=0,\text{ for any }x\in L_{\overline{1}}.$$

A Lie superalgebra is said to be \emph{abelian} if $[x,y]=0$ for all $x,y\in L$.
The \emph{dimension} of a Lie superalgebra $L$ is an ordered pair $(d_{0},d_{1})$,
where $d_{0}=\operatorname{dim}_{\mathbb{F}}L_{\overline{0}}$ and $d_{1}=\operatorname{dim}_{\mathbb{F}}L_{\overline{1}}$.
A $\mathbb{Z}_{2}$-graded subspace $I=I_{\overline{0}}\oplus I_{\overline{1}}$ of $L$ is called a \emph{graded ideal}
if $[x,y]\in I_{\overline{l+k}}$ for all $x\in I_{\overline{l}},y\in L_{\overline{k}}$,
where $\overline{l},\overline{k}\in \mathbb{Z}_{2}$ .
The intersections, sums, and Lie brackets of graded ideals are graded ideals.

The \emph{derived series} of $L$ is defined recursively by
\begin{equation*}
L^{(0)}=L;\ \ L^{(k+1)}=[L^{(k)},L^{(k)}],\ k\in \mathbb{Z}.
\end{equation*}

The \emph{descending central series} is given by
\begin{equation*}
L^{1}=L;\ \ L^{k+1}=[L^{k},L],\ k\in \mathbb{Z}.
\end{equation*}

A Lie superalgebra $L$ is called \emph{solvable} if $L^{(r)}=0$ for some $r\in \mathbb{N}$, and \emph{nilpotent} if $L^{k}=0$ for some $k\in \mathbb{N}$.
 Every nilpotent Lie superalgebra is solvable, but the converse is not true.

For $x,y\in L$, define the \emph{adjoint map} $\operatorname{ad}_{x}:\ L\rightarrow L$ by $\operatorname{ad}_{x}(y)=[x,y]$.
The map $\operatorname{ad}:\ L\rightarrow\ \mathfrak{gl}(L)$, where $\mathfrak{gl}(L)$ is the set of endomorphisms from $L$ to $L$,
which is called the \emph{adjoint representation} of $L$.

The center of $L$ is $Z(L)=\{x\in L|[x,L]=0\}$. Obviously, $Z(L)=Z(L)_{\overline{0}}\oplus Z(L)_{\overline{1}}$,
where $Z(L)_{\overline{i}}=Z(L)\cap L_{\overline{i}},\ \overline{i}\in \mathbb{Z}_{2}$. Define the \emph{upper central series} recursively by
\begin{equation*}
Z_{0}(L)=0,
\end{equation*}
\begin{equation*}
Z_{i+1}(L)=\{x\in L| [x,L]\subseteq Z_{i}(L)\},\ i\geq 0.
\end{equation*}
The union $\bigcup_{i\geq 0}Z_{i}(L)$ is called the \emph{hypercenter}, denoted by $Z^{*}(L)$. For a finite dimensional Lie superalgebra $L$, this series terminates.
Clearly, $Z^{*}(L)$ is a graded ideal of $L$.

\section{The nilpotentizer of a Lie superalgebra}\label{section3}
\ \ \ \ \ \ In this section, we present some properties of nilpotentizers in Lie superalgebras.
Throughout, subalgebra refers to a $\mathbb{Z}_{2}$-graded subalgebra, and $\langle z,x\rangle$ denotes the subalgebra of $L$ generated by $z$ and $x$.
\begin{definition}
Let $L$ be a finite dimensional Lie superalgebra. Then
\begin{enumerate}[label=(\arabic*),font=\textnormal]
\item The nilpotentizer of $z$ in $L$ is defined by:
\begin{equation*}
\operatorname{nil}_{L}(z):= \{x\in L|\langle z,x\rangle\ is\ a\ nilpotent\ subalgebra\ of\ L\}.
\end{equation*}

\item The nilpotentizer of $L$ is defined by:
\begin{equation*}
\operatorname{nil}(L):= \{x\in L|\langle z,x\rangle\ is\ a\ nilpotent\ subalgebra\ of\ L, \ for\ all\ z\in L\}.
\end{equation*}
\end{enumerate}
\end{definition}

Let $E_{ij}$ be an $n\times n$ matrix with $1$ in the $(i,j)$-th entry and $0$ elsewhere.

\begin{example}\label{example 1}
Suppose that $\mathfrak{gl}(1|1)$ is the general linear Lie superalgebra with the basis $\{E_{11},E_{22}|E_{21},E_{12}\}$ and $\operatorname{char} \mathbb{F}=0$. Then
\begin{enumerate}[label=(\arabic*),font=\textnormal]
\item $\operatorname{nil}_{\mathfrak{gl}(1|1)}(E_{11}-E_{22})=\mathfrak{gl}(1|1)_{\overline{0}}$;

\item $\operatorname{nil}_{\mathfrak{gl}(1|1)}(E_{11}+E_{12})=\varnothing$;

\item $\operatorname{nil}(\mathfrak{gl}(1|1))=\varnothing$.
\end{enumerate}
\end{example}
\begin{proof}
(1) Let $z = E_{11} - E_{22}$. Compute the Lie brackets of $z$ with the basis elements:
\begin{align*}
[z, E_{11}] = 0;\
[z, E_{22}] = 0;\
[z, E_{12}] = 2E_{12};\
[z, E_{21}] = -2E_{21}.
\end{align*}

For any $x = a E_{11} + b E_{22}\in \mathfrak{gl}(1|1)_{\overline{0}}$, where $a, b, c, d \in \mathbb{F}$.
The subalgebra $\langle z, x \rangle$ is abelian and hence nilpotent. Therefore, $x\in \operatorname{nil}_{\mathfrak{gl}(1|1)}(z)$.
This shows that $\mathfrak{gl}(1|1)_{\overline{0}}\subseteq \operatorname{nil}_{\mathfrak{gl}(1|1)}(z)$.

Conversely, suppose $x\notin \operatorname{nil}_{\mathfrak{gl}(1|1)}(z)$.
Then $x = a E_{11} + b E_{22} + c E_{12} + d E_{21}$ with $c \neq 0$ or $d \neq 0$.
Consider the subalgebra $\langle z, x \rangle$.
Since $\operatorname{ad}_z$ has eigenvalues $\pm 2$ on the odd part, it is not nilpotent,
implying that $\langle z, x \rangle$ is not a nilpotent subalgebra. Thus, $x\notin \operatorname{nil}_{\mathfrak{gl}(1|1)}(z)$.
In conclusion, $\operatorname{nil}_{\mathfrak{gl}(1|1)}(z) = \mathfrak{gl}(1|1)_{\overline{0}}.$

(2) For any $x \in \mathfrak{gl}(1|1)$, the elements $E_{11}$ and $E_{12}$ are contained in $\langle E_{11}+E_{12}, x \rangle$.
Given that $\operatorname{ad}_{E_{11}}^{n}(E_{12}) = E_{12}$ for all $n \in \mathbb{N}$, so this subalgebra is not nilpotent. Therefore, $\operatorname{nil}_{\mathfrak{gl}(1|1)}(E_{11}+E_{12}) = \varnothing$.
\noindent

(3) By definition, $\operatorname{nil}(\mathfrak{gl}(1|1))=\varnothing$.
\end{proof}

\noindent \textbf{Remark.}
(1) For a Lie superalgebra, the center and $Z^{*}(L)$ are not necessarily contained in its nilpotentizer.
An example is provided in Example \ref{example 1}, where
$x = E_{11} + E_{22}\in Z(\mathfrak{gl}(1|1))\subseteq Z^{*}(\mathfrak{gl}(1|1))$, yet $x\notin \operatorname{nil}(\mathfrak{gl}(1|1))$.
However, in a Lie algebra, the center is contained in its nilpotentizer.

(2) For a Lie superalgebra, an element is not necessarily contained in its nilpotentizer.
An example is provided by  Example \ref{example 1},
$E_{11}+E_{12} \notin \operatorname{nil}_{\mathfrak{gl}(1|1)}(E_{11}+E_{12})$.
In contrast, in a Lie algebra, any element $z$ is always contained in its nilpotentizer, since $[z,z]=0$.

\begin{proposition}\label{lemma4}
Let $L=L_{\overline{0}}\oplus L_{\overline{1}}$ be a finite dimensional Lie superalgebra, and let $J$ be a graded ideal of $L$ and $x,z\in L$. Then
\begin{enumerate}[label=(\arabic*),font=\textnormal]
\item $\operatorname{nil}(L)=\cap_{z\in L}\operatorname{nil}_{L}(z)$;

\item If $L$ is nilpotent, then $\operatorname{nil}(L)=L$;

\item $\operatorname{nil}_{L/J}(z+J)=(\operatorname{nil}_{L}(z)+J)/J$ where $J\subseteq Z^{*}(L)$;

\item $\operatorname{nil}_{L}(z)$ is the union of all maximal nilpotent subalgebras containing $z$.
\end{enumerate}
\end{proposition}

\begin{proof}
(1) and (2) are straightforward.

(3) On the one hand, for $x\in \operatorname{nil}_{L}(z)$, the subalgebra $\langle z, x \rangle$ is nilpotent.
Furthermore, $\langle z+J,x+J\rangle=\langle z,x\rangle+J$, it follows that $x+J\in\operatorname{nil}_{L/J}(z+J)$.
On the other hand, suppose $x+J \in \operatorname{nil}_{L/J}(z+J)$.
Then the subalgebra $\langle z+J, x+J \rangle$ is nilpotent in $L/J$,
meaning there exists $n \in \mathbb{N}$ such that $\langle z+J, x+J \rangle^n = 0$ in $L/J$.
This implies that $\langle z, x \rangle^n \subseteq J \subseteq Z^*(L)$,
it follows that the subalgebra $\langle z, x \rangle$ is nilpotent in $L$.
Therefore, $x \in \operatorname{nil}_{L}(z)$, which means $x+J \in (\operatorname{nil}_{L}(z)+J)/J$.

(4) For any $x \in \operatorname{nil}_{L}(z)$,
the subalgebra $\langle z, x \rangle$ is nilpotent and therefore is contained in some maximal nilpotent subalgebra of $L$ that contains $z$.
Now, let $M_{z}$ be a maximal nilpotent subalgebra of $L$ that contains $z$.
For any $y \in M_{z}$, the subalgebra $\langle y, z \rangle$ is contained in $M_{z}$ and is therefore nilpotent. This implies $y \in \operatorname{nil}_{L}(z)$.
We conclude that $M_{z} \subseteq \operatorname{nil}_{L}(z)$.
\end{proof}

Denote by $\mathrm{Der}L$ the derivation algebra of $L$. Let $x\in L$, put
\begin{equation*}
E_{L}(x)=\{y\in L|(\operatorname{ad}_{x})^{n}(y)=0,\text{ for some } n\in \mathbb{N}\}.
\end{equation*}
We say that $y$ is an \emph{Engel element} of $L$ if $y\in E_{L}(x)$ for all $x\in L$.
Let $E(L)$ be the set of all Engel elements of $L$.
Obviously, $E_{L}(x)$ and $E(L)$ are subspaces of $L$.

We recall the the Uniqueness Theorem for Analytic Functions.
Let $\mathbb{C}$ be the field of complex numbers, and let $V$ be a finite dimensional vector space over $\mathbb{C}$.
Consider an analytic function $f$ defined by the power series
$$f(t) = \sum_{k=0}^{\infty} a_{k} t^{k}, \quad (a_{k} \in V),$$
which converges in some neighborhood of $t=0$.
If $f(t)$ vanishes on a subset of $\mathbb{C}$ that has a limit point, then $a_k = 0$ for all $k$.

\begin{theorem}\label{theorem1}
Suppose that $L$ is a finite dimensional Lie superalgebra over $\mathbb{C}$. Then $D(\operatorname{nil}(L))\subseteq E(L)$ for every $D\in \operatorname{Der}(L)$.
\end{theorem}
\begin{proof}
Firstly, we prove that $E(L)$ is invariant under automorphisms of $L$.
Let $\rho$ be an automorphism of $L$. For any $y\in E(L)$ and any $x\in L$, there exists $n_{x,y}$ such that $(\operatorname{ad}_{x})^{n_{x,y}}(y)=0$,
and $(\operatorname{ad}_{\rho (x)})^{n_{x,y}}(\rho (y))=\rho\circ (ad_{x})^{n_{x,y}}\circ \rho^{-1}\circ \rho (y)=\rho(0)=0$.
Then $\rho(y)\in E(\rho(L))$ and $E(L)$ is invariant under automorphisms of $L$.

Next, we will prove $E(L)$ is invariant under derivations of $L$.
Let $\operatorname{Exp}$ be the exponential map on $\mathrm{Der}L$.
For any $D\in \mathrm{Der}L$, since $L$ is finite dimensional,
$\operatorname{Exp}(tD)=\sum_{k=0}^{\infty}\frac{t^{k}}{k!}D^{k}$, $(t\in \mathbb{C})$ converges.
It follows from the properties of the exponential map that $\operatorname{Exp}(tD)$ is an automorphism.
Therefore, $\operatorname{Exp}(tD)(E(L))\subseteq E(L)$ for all $t\in \mathbb{C}$.
Let $\pi: L\rightarrow L/E(L)$ be the quotient map. Then $\pi(\operatorname{Exp}(tD)(y))=0$ for all $t\in \mathbb{C}$.
This shows that $\pi(\operatorname{Exp}(tD)(y))=\sum_{k=0}^{\infty}\frac{t^{k}}{k!}\pi(D^{k}(y))$ vanishes for infinitely many $t$.
However, $L/E(L)$ is finite dimensional. By the Uniqueness Theorem for Analytic Functions, $\pi(D(y))=0$. Hence $D(y)\in E(L)$.

Finally, let $z\in \operatorname{nil}(L)$. For any $x\in L$, the subalgebra $\langle z,x\rangle$ is nilpotent, which implies $z\in E_{L}(x)$.
Since $x$ is arbitrary, it follows that $z\in E(L)$. We have already shown that $E(L)$ is invariant under any derivation $D$. Therefore,
$D(z)\in E(L)$ for all $z\in \operatorname{nil}(L)$, which means $D(\operatorname{nil}(L))\subseteq E(L)$.
\end{proof}

Let $\langle x_{1},\ldots,x_{n}\rangle_{L_{1}}$ denote the subalgebra of $L_{1}$ generated by $x_{1},\ldots,x_{n}$.
We now present the Extension Theorem for the nilpotentizer.

\begin{theorem}
Let $L$ be a finite dimensional Lie superalgebra and $\mathfrak{g}$ a subalgebra of $L$.
Then we have
$$\operatorname{nil}_{L}(x)\cap \mathfrak{g}=\operatorname{nil}_{\mathfrak{g}}(x),\text{ for any }x\in \mathfrak{g}.$$
\end{theorem}
\begin{proof}
On the one hand, for any $y\in \operatorname{nil}_{L}(x)\cap \mathfrak{g}$, we have $y\in L$ and the subalgebra $\langle x,y\rangle_{L}$ is nilpotent.
Hence $\langle x,y\rangle_{\mathfrak{g}}=\langle x,y\rangle_{L}$ is also nilpotent, which means $y\in \operatorname{nil}_{\mathfrak{g}}(x)$.

On the other hand, for any $y\in \operatorname{nil}_{\mathfrak{g}}(x)$,
the subalgebra $\langle x,y\rangle_{\mathfrak{g}}$ is nilpotent in $\mathfrak{g}$. Since $\mathfrak{g}\subseteq L$,
it follows that $\langle x,y\rangle_{\mathfrak{g}}$ is a nilpotent subalgebra of $L$, which means $y\in \operatorname{nil}_{L}(x)\cap \mathfrak{g}$.
\end{proof}

Next, we recall the definition of the direct sum of Lie superalgebras.
Let $L = L_{\overline{0}} \oplus L_{\overline{1}}$ and $M = M_{\overline{0}} \oplus M_{\overline{1}}$ be Lie superalgebras over a field $\mathbb{F}$.
The direct sum of $L$ and $M$, denoted by $L \oplus M$, is the Lie superalgebra defined as follows:
\begin{enumerate}[label=(\arabic*)]
    \item  As a vector space, $L \oplus M$ is the direct sum of $L$ and $M$;
    \item $(L \oplus M)_{\alpha} = L_{\alpha} \oplus M_{\alpha}$ for $\alpha \in \mathbb{Z}_{2}$;
    \item For arbitrary $x_{1}, x_{2} \in L$ and $y_{1}, y_{2} \in M$,
    \[
    [x_1 + y_1, x_2 + y_2] = [x_1, x_2]_{L} + [y_1, y_2]_{M},
    \]
    where $[-, -]_{L}$ and $[-, -]_{M}$ denote the super brackets in $L$ and $ M$ respectively, with $[L, M] = \{0\}$.
\end{enumerate}

\begin{lemma}\label{lemma8}
Suppose that $L_{1}$ and $L_{2}$ are two Lie superalgebras. For any
$l_{1},l_{2}\in L_{1}$ and $m_{1},m_{2}\in L_{2}$, the subalgebra $\langle (l_{1},m_{1}),(l_{2},m_{2})\rangle$ is nilpotent
in $L_{1}\oplus L_{2}$ if and only if $\langle l_{1},l_{2}\rangle$ and $\langle m_{1},m_{2}\rangle$ are nilpotent subalgebras of $L_{1}$ and $L_{2}$, respectively.
\end{lemma}
\begin{proof}
The graded structure of Lie superalgebras does not affect the direct sum decomposition; the result follows immediately by analogy to  Lemma 5.1 in \cite{TDIL}.
\end{proof}

\begin{theorem}\label{theorem4}
Let $L=L_{1}\oplus L_{2}$ be the direct sum of finite dimensional Lie superalgebras. Then the following properties hold:
\begin{enumerate}[label=(\arabic*),font=\textnormal]
\item $\operatorname{nil}(L)=\operatorname{nil}(L_{1})\oplus \operatorname{nil}(L_{2}).$

\item For any $x = (x_{1},x_{2}) \in L$, we have
$$\operatorname{nil}_L(x)=\operatorname{nil}_{L_1}(x_1) \oplus\operatorname{nil}_{L_2}(x_2).$$
\end{enumerate}
\end{theorem}
\begin{proof}
(1) By Lemma \ref{lemma8}, we have the following chain of equivalences:
$$
\begin{aligned}
  &(y_1, y_2) \in \mathrm{nil}(L_1) \oplus \mathrm{nil}(L_2)\\
  \Leftrightarrow &\langle y_1, x_1 \rangle_{L_1} \text{ and } \langle y_2, x_2 \rangle_{L_2} \text{ are nilpotent} , \text{ for any } x_{1}\in L_{1}, x_{2}\in L_{2}\\
  \Leftrightarrow &\langle (y_1, y_2), (x_1, x_2) \rangle_{L_1 \oplus L_2} \text{ is nilpotent} , \text{ for any } x=(x_{1},x_{2})\in L_{1}\oplus L_{2}\\
  \Leftrightarrow &(y_1, y_2) \in \mathrm{nil}(L_1 \oplus L_2).
\end{aligned}
$$
Hence $\operatorname{nil}(L_{1}\oplus L_{2})=\operatorname{nil}(L_{1})\oplus \operatorname{nil}(L_{2}).$

(2) The conclusion follows directly from (1).
\end{proof}
Proceeding by induction, suppose that for a direct sum of finite-dimensional Lie superalgebras
$L=L_{1}\oplus L_{2}\oplus\cdots\oplus L_{n}$, we have
$$\operatorname{nil}(L)=\operatorname{nil}(L_{1})\oplus \operatorname{nil}(L_{2})\oplus\cdots\oplus \operatorname{nil}(L_{n})$$
and for any $x = (x_1,\ldots,x_n) \in L$, we have
$$\operatorname{nil}_L(x)=\operatorname{nil}_{L_1}(x_1) \oplus\cdots\oplus \operatorname{nil}_{L_n}(x_n).$$

Let $L$ be a Lie superalgebra and $L_{1}$ be a subalgebra of $L$.
For $h\in L$, let $\operatorname{ad}_{L_{1}}(h)$ denote the adjoint representation of $h$ in $L_{1}$.
A homomorphism of Lie superalgebras is, by definition,
a linear map that is both a homomorphism of the underlying algebras and preserves the $\mathbb{Z}_{2}$-grading.
In particular, it is a homogeneous map of degree $0$. The same applies to the definitions of isomorphisms and automorphisms.

\begin{theorem}\label{12345}
Suppose that $L_{1},L_{2}$ are two finite dimensional Lie superalgebras with char $\mathbb{F}=0$,
and $\varphi: L_{1}\longrightarrow L_{2}$ is a homomorphism.
Then:
\begin{enumerate}[label=(\arabic*),font=\textnormal]
\item$\varphi(\operatorname{nil}(L_{1}))\subseteq \operatorname{nil}(\varphi(L_{1})).$

\item If $\operatorname{ker} \varphi\subseteq \operatorname{nil}(L_{1})$, then $\varphi(\operatorname{nil}(L_{1}))=\operatorname{nil}(\varphi(L_{1}))$.
\end{enumerate}
\end{theorem}
\begin{proof}

(1) For all $x\in \operatorname{nil}(L_{1})$. By the definition of the nilpotentizer, for any $z\in L_{1}$, the subalgebra $\langle x,z\rangle$ is nilpotent.
Since $\varphi$ is a homomorphism, the image $\langle \varphi(x),\varphi(z)\rangle$ is also a nilpotent subalgebra for any $\varphi(z)\in \varphi(L)$.
Therefore, $\varphi(x)\in \operatorname{nil}(\varphi(L_{1}))$, which proves that $\varphi(\operatorname{nil}(L_{1}))\subseteq \operatorname{nil}(\varphi(L_{1}))$.

(2) By (1), it suffices to prove that $\operatorname{nil}(\varphi(L_{1}))\subseteq \varphi(\operatorname{nil}(L_{1}))$.
Suppose that $\operatorname{ker} \varphi\subseteq \operatorname{nil}(L_{1})$.
For any $\varphi(x)\in \operatorname{nil}(\varphi(L_{1}))$, we have
$$\langle \varphi(x), \varphi(z)\rangle=\varphi(\langle x,z\rangle),\text{ for any }z\in L_{1}$$
which is a nilpotent subalgebra.
To conclude that $x\in \operatorname{nil}(L_{1})$,
we need to show that $\langle x,z\rangle$ is nilpotent.
By Engel's Theorem for Lie superalgebras, it suffices to prove that $\operatorname{ad}_{\langle x,z\rangle}(h)$ is nilpotent for every $h\in \langle x,z\rangle$.

It is obvious that $\langle x,z\rangle$ is a finite dimensional subalgebra. Without loss of generality, assume $\operatorname{dim} \langle x,z\rangle=n$.
For any $h\in \langle x,z\rangle$,
it can be observed that $$\operatorname{ad}_{\langle x,z\rangle}(h)(\operatorname{Im}((\operatorname{ad}_{\langle x,z\rangle}(h))^{n}))=
\operatorname{Im}((\operatorname{ad}_{\langle x,z\rangle}(h))^{n}),$$
and
$$(\operatorname{ad}_{\langle x,z\rangle}(h))^{n}(\operatorname{ker}(\operatorname{ad}_{\langle x,z\rangle}(h))^{n}))=0.$$
By the Fitting decomposition with respect to $\operatorname{ad}_{\langle x,z\rangle}(h)$, we have:

$$\langle x,z\rangle=\operatorname{Im}((\operatorname{ad}_{\langle x,z\rangle}(h))^{n}) \oplus \operatorname{ker}((\operatorname{ad}_{\langle x,z\rangle}(h))^{n}).$$

Let $N= \operatorname{Im}((\operatorname{ad}_{\langle x,z\rangle}(h))^{n})$.
Since $\varphi$ is a homomorphism, for any $r,l \in L$ and $k\in \mathbb{Z}$,
the identity $$\varphi(\operatorname{ad}_{r}^{k}(l))=\operatorname{ad}_{\varphi(r)}^{k}(\varphi(l))$$ holds.
For any $v\in N$, there exists a $v'\in \langle x,z\rangle$ such that $v=(\operatorname{ad}_{\langle x,z\rangle}(h))^{n}(v')$. Thus

$$\varphi(v)=\varphi((\operatorname{ad}_{\langle x,z\rangle}(h))^{n}(v'))=(\operatorname{ad}_{\langle \varphi(x),\varphi(z)\rangle}(\varphi(h))^{n}(\varphi(v')).$$

Since $\varphi(\langle x,z\rangle)$ is nilpotent, then there exists some $s\in \mathbb{N}$ such that
$$\operatorname{ad}_{\varphi(\langle x,z\rangle)}(\varphi(h))^{s}=0.$$
We consider two cases:

(I) If $s\leq n$, $\operatorname{ad}_{\varphi(\langle x,z\rangle)}(\varphi(h))^{n}=0$, so $\varphi(v)=0$.

(II) If $s>n$, since $\operatorname{ad}_{\langle x,z\rangle}(h)(N)=N$, there exists a $v''\in N$ such that
$v=(\operatorname{ad}_{\langle x,z\rangle}(h))^{s}(v'')$.
Then $$\varphi(v)=(\operatorname{ad}_{\langle \varphi(x),\varphi(z)\rangle}(\varphi(h)))^{s}(\varphi(v''))=0.$$

Therefore, $N\subseteq \langle x,z\rangle \cap \operatorname{ker}(\varphi)\subseteq \operatorname{nil}(L_{1})$.
Since $\operatorname{ad}_{N}(h)$ is nilpotent and surjective,
it must be that $\operatorname{ad}_{N}(h)=0$. But, $N= \operatorname{Im}((\operatorname{ad}_{\langle x,z\rangle}(h))^{n})$.
So for any $u\in \langle x,z\rangle$,
there exists $u'= \operatorname{ad}_{\langle x,z\rangle}(h)^{n}(u)$ such that $$\operatorname{ad}_{\langle x,z\rangle}(h)^{n+1}(u)=[h,u']=0.$$
Thus, $\operatorname{ad}_{\langle x,z\rangle}(h)$ is nilpotent.

Due to the arbitrariness of $h\in \langle x,z\rangle$, and by Engel's Theorem for Lie superalgebras,
$\langle x,z\rangle$ is a nilpotent subalgebra. Therefore, $x\in \operatorname{nil}(L_{1})$, which implies $\operatorname{nil}(\varphi(L_{1}))\subseteq \varphi(\operatorname{nil}(L_{1}))$.

In conclusion, $\varphi(\operatorname{nil}(L_{1}))=\operatorname{nil}(\varphi(L_{1}))$.
\end{proof}

\begin{corollary}\label{co}
Suppose that $L=L_{1}\oplus\cdots \oplus L_{n}$ and $M=M_{1}\oplus\cdots\oplus M_{n}$
are two direct sums of finite dimensional Lie superalgebras with char $\mathbb{F}=0$,
and let
$\varphi=\varphi_{1}+ \cdots+ \varphi_{n}: L\longrightarrow M$ be a homomorphism,
where $\varphi_{i}: L_{i}\longrightarrow M_{i}$ $(1\leq i \leq n)$ is a homomorphism and $\varphi_{i}(L_{j})=0$ for $i \neq j$. If $\operatorname{ker}\ \varphi\subseteq \operatorname{nil}(L)$,
then $$\operatorname{nil}(\varphi(L))=\varphi_{1}(\operatorname{nil}(L_{1}))\oplus\cdots\oplus \varphi_{n}(\operatorname{nil}(L_{n})).$$
\end{corollary}
\begin{proof}
By Theorems \ref{theorem4} and \ref{12345}, we have
$$\operatorname{nil}(\varphi(L))=\varphi(\operatorname{nil}(L))=
\varphi(\operatorname{nil}(L_{1})\oplus \cdots\oplus \operatorname{nil}(L_{n}))=
\varphi_{1}(\operatorname{nil}(L_{1}))\oplus \cdots\oplus \varphi_{n}(\operatorname{nil}(L_{n})).$$
\end{proof}
\section{The nilpotent graph of a Lie superalgebra}\ \ \ \
In this section, we describe the nilpotentizer $\operatorname{nil}(L)$ from the perspectives of graph theory.
A \emph{graph} $\Gamma = (V, E)$ consists of a non-empty vertex set $V$ and an edge set $E$,
where $E$ is a collection of unordered pairs of distinct elements from $V$. If $\{u, v\} \in E$,
then $u$ and $v$ are said to be \emph{adjacent}.
A \emph{subgraph} of $\Gamma$ is a graph $\Gamma' = (V', E')$ with $\varnothing \neq V' \subseteq V$ and $E' \subseteq E$.
A subgraph $\Gamma'$ is called a \emph{clique} if every pair of distinct vertices is adjacent.
In particular, if $\Gamma$ is a clique, then it is called a \emph{complete graph}. The complete graph on $n$ vertices is denoted by $K_{n}$.

A \emph{path} in $\Gamma$ is a sequence of distinct vertices $v_{1}, \dots, v_{k}$ such that $\{v_{i}, v_{i+1}\} \in E$ for all $1 \leq i < k$.
The graph $\Gamma$ is \emph{connected} if there exists a path between any two distinct vertices;
otherwise, it is \emph{disconnected}, and its maximal connected subgraph is called \emph{component};
the number of components is denoted by $\kappa(\Gamma)$.
For more details, see \cite{RDi}.


\begin{definition}\label{defgraphnil}
Let $L$ be a finite dimensional non-nilpotent Lie superalgebra over a field $\mathbb{F}$.
Its nilpotent graph $\mathfrak{G}(L)=(V,E)$ is defined as follows:
\begin{enumerate}[label=(\arabic*),font=\textnormal]
\item The vertex set $V=L\setminus (\operatorname{nil}(L)\cup \{0\})$.

\item For two distinct vertices $x,y\in V$, $\{x,y\}\in E$ if and only if the subalgebra $\langle x,y\rangle$ is nilpotent.
\end{enumerate}
\end{definition}

\begin{definition}\label{defgraphnil2}
Let $L$ be a finite dimensional non-nilpotent Lie superalgebra over a field $\mathbb{F}$.
Its non-nilpotent graph $\mathfrak{G}(L)^{c}=(V,E)$ is defined as follows:
\begin{enumerate}[label=(\arabic*),font=\textnormal]
\item The vertex set $V=L\setminus (\operatorname{nil}(L)\cup \{0\})$.

\item For two distinct vertices $x,y\in V$, $\{x,y\}\in E$ if and only if the subalgebra $\langle x,y\rangle$ is non-nilpotent.
\end{enumerate}
\end{definition}

Obviously, $\mathfrak{G}(L)\cup \mathfrak{G}(L)^{c}$ forms the complete graph.

\begin{example}\label{example}
Let $L=L_{\overline{0}}\oplus L_{\overline{1}}$ be a Lie superalgebra with char $\mathbb{F}=3$, where $L_{\overline{0}}=span\{x\}$ and
$L_{\overline{1}}=span\{y\}$. The Lie bracket is defined by
$$[y,y]=0,\ [x,y]=y.$$
The nilpotent graph and non-nilpotent graph of $L$ are as follows:
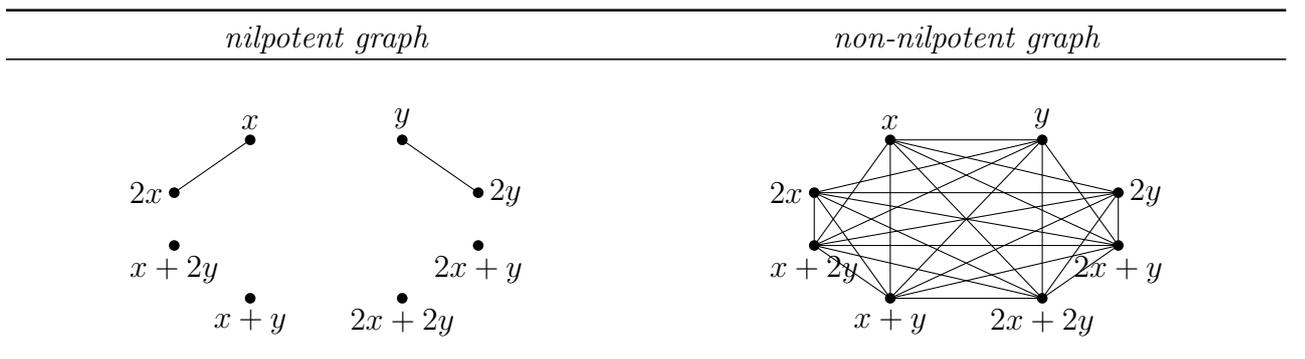
\begin{longtable}{m{8cm}<{\centering} m{8cm}<{\centering} m{8cm}<{\centering}}
        \toprule
         nilpotent graph &  non-nilpotent graph \\ \hline
        \endhead
        \bottomrule
        \endfoot
        \bottomrule
        \endlastfoot
       \\         \begin{tikzpicture}
        \fill(0,0)circle(2pt)coordinate(dot);
         \node[above] at (dot){$x$};
         \fill(2,0)circle(2pt)coordinate(dot);
         \node[above] at (dot){$y$};
         \fill(-1,-0.7)circle(2pt)coordinate(dot);
         \node[left] at (dot){$2x$};
         \fill(3,-0.7)circle(2pt)coordinate(dot);
         \node[right] at (dot){$2y$};
         \fill(-1,-1.4)circle(2pt)coordinate(dot);
         \node[below] at (dot){$x+2y$};
         \fill(3,-1.4)circle(2pt)coordinate(dot);
         \node[below] at (dot){$2x+y$};
         \fill(0,-2.1)circle(2pt)coordinate(dot);
         \node[below] at (dot){$x+y$};
         \fill(2,-2.1)circle(2pt)coordinate(dot);
         \node[below] at (dot){$2x+2y$};

         \draw(0,0)--(-1,-0.7);
         \draw(2,0)--(3,-0.7);

        \end{tikzpicture} &
        \begin{tikzpicture}
         \fill(6,0)circle(2pt)coordinate(dot);
         \node[above] at (dot){$x$};
         \fill(8,0)circle(2pt)coordinate(dot);
         \node[above] at (dot){$y$};
         \fill(5,-0.7)circle(2pt)coordinate(dot);
         \node[left] at (dot){$2x$};
         \fill(9,-0.7)circle(2pt)coordinate(dot);
         \node[right] at (dot){$2y$};
         \fill(5,-1.4)circle(2pt)coordinate(dot);
         \node[below] at (dot){$x+2y$};
         \fill(9,-1.4)circle(2pt)coordinate(dot);
         \node[below] at (dot){$2x+y$};
         \fill(6,-2.1)circle(2pt)coordinate(dot);
         \node[below] at (dot){$x+y$};
         \fill(8,-2.1)circle(2pt)coordinate(dot);
         \node[below] at (dot){$2x+2y$};

         \draw(6,0)--(8,0);
         \draw(6,0)--(9,-0.7);
         \draw(6,0)--(5,-1.4);
         \draw(6,0)--(9,-1.4);
         \draw(6,0)--(6,-2.1);
         \draw(6,0)--(8,-2.1);

         \draw(8,0)--(5,-0.7);
         \draw(8,0)--(5,-1.4);
         \draw(8,0)--(9,-1.4);
         \draw(8,0)--(6,-2.1);
         \draw(8,0)--(8,-2.1);

         \draw(5,-0.7)--(9,-0.7);
         \draw(5,-0.7)--(5,-1.4);
         \draw(5,-0.7)--(9,-1.4);
         \draw(5,-0.7)--(6,-2.1);
         \draw(5,-0.7)--(8,-2.1);

         \draw(9,-0.7)--(5,-1.4);
         \draw(9,-0.7)--(6,-2.1);
         \draw(9,-0.7)--(8,-2.1);
         \draw(9,-0.7)--(9,-1.4);

         \draw(5,-1.4)--(9,-1.4);
         \draw(5,-1.4)--(6,-2.1);
         \draw(5,-1.4)--(8,-2.1);

         \draw(9,-1.4)--(6,-2.1);
         \draw(9,-1.4)--(8,-2.1);

         \draw(6,-2.1)--(8,-2.1);
                \end{tikzpicture}
                 \\
\end{longtable}
\end{example}
\begin{proof}
Since $\operatorname{nil}(L)=\varnothing$, by Definitions \ref{defgraphnil} and \ref{defgraphnil2},
the vertex set for the nilpotent graph and non-nilpotent graph are
$$V=\{x, y, 2x,2y,x+2y,2y+x,x+y,2x+2y\}.$$
We now analyze the adjacency condition. For any integer $n> 0$, we compute $\operatorname{ad}_{x}^{n} (y)=y$.
It follows that for any $a,b\in L$, the subalgebra $\langle a,b\rangle$ is nilpotent if and only if
$x\notin \langle a,b\rangle$ or $y\notin \langle a,b\rangle$.
This completes the analysis of this example.
\end{proof}

\begin{proposition}\label{theorem5}
Suppose that $L$ is a finite dimensional non-nilpotent Lie superalgebra.
Then for every finite dimensional nilpotent Lie superalgebra $L_{1}$,
the number of connected components of their nilpotent graphs satisfies $$\kappa(\mathfrak{G}(L\oplus L_{1}))\leq \kappa(\mathfrak{G}(L)).$$
In particular, If $\mathfrak{G}(L)$ is connected, then $\mathfrak{G}(L\oplus L_{1})$ is connected.
\end{proposition}
\begin{proof}
Since $L_{1}$ is nilpotent, we have $\operatorname{nil}(L_{1})=L_{1}$.
By Theorem \ref{theorem4}, it follows that $$\operatorname{nil}(L\oplus L_{1})=\operatorname{nil}(L)\oplus L_{1}.$$
Now, consider the nilpotent graphs. By Lemma \ref{lemma8},
two vertices $(z,w)$ and $(x,y)$ in $\mathfrak{G}(L\oplus L_{1})$ are adjacent if and only if $z$ and $x$ are adjacent in $\mathfrak{G}(L)$.
This implies that for a path $z \to \cdots \to x$ in
$\mathfrak{G}(L)$ and for any $w, y \in L_{1}$, there exists a path $(z, w) \to \cdots \to (x, y)$ exists in $\mathfrak{G}(L \oplus L_1)$.
Consequently, the number of connected components is preserved, i.e.,
\begin{equation}\label{=1}
\kappa(\mathfrak{G}(L \oplus L_1)) \leq \kappa(\mathfrak{G}(L)).
\end{equation}
\end{proof}

\noindent \textbf{Remark.}
 The equality does not hold in Formula \ref{=1} . For example, take $L$ in Example \ref{example},
 then $\kappa(\mathfrak{G}(L))=2$ and $\kappa(\mathfrak{G}(L\oplus L))=1$ by Theorem 7.1 in \cite{HAS}.

\vspace{1em}
Next, we introduce the nilpotency measure of a Lie superalgebra to characterize its nilpotent structure.
\begin{definition}
Let $L$ be a finite dimensional non-nilpotent Lie superalgebra with char $\mathbb{F}\geq3$ and
$\mathfrak{G}(L)=(V,E)$ be its nilpotent graph.
We define the nilpotency measure $\mu(L)$ of $L$ as:
$$\mu(L) = 1 - \frac{2 |E|}{|V|(|V| - 1)},$$
where $|E|$ and $|V|$ denote the number of elements in $E$ and $V$, respectively.
\end{definition}
Since $|V|\geq 2$, the definition is well-defined.

\begin{example}
The nilpotency measure of the Lie superalgebra in Example\ref{example} is $\frac{13}{14}$.
\end{example}

\begin{proposition}\label{nilpotency measure}
Let $L$ be a finite dimensional non-nilpotent Lie superalgebra with $\operatorname{char}\mathbb{F} \geq 3$.
Then the following properties hold:
\begin{enumerate}[label=(\arabic*),font=\textnormal]
    \item $0 \leq \mu(L) \leq 1$.
    \item $\mu(L) = 0$ if and only if $\langle x,y \rangle$ is nilpotent for all $ x,y\in V$.
    \item $\mu(L) = 1$ if and only if $\langle x, y\rangle$ is non-nilpotent for every pair of distinct vertices $x,y\in V$.
\end{enumerate}
\end{proposition}
\begin{proof}
(1) Suppose $\mathfrak{G}(L)= (V,E)$ is the nilpotent graph of $L$.
Since $E$ is a subset of the set of all unordered pairs from $V$, we have $0 \leq |E| \leq \binom{|V|}{2} = \frac{|V|(|V| - 1)}{2}$. Therefore,
$$0 \leq \mu(L)=1-\frac{2 |E|}{|V|(|V| - 1)} \leq 1.$$

(2) Assume that for every pair of distinct elements $x$ and $y$ in $V$, the subalgebra $\langle x, y\rangle$ is nilpotent,
then $E$ contains all possible unordered pairs, that is $|E| = \binom{|V|}{2} = \frac{|V|(|V| - 1)}{2}$. Therefore,
$$\mu(L) = 1 - \frac{2 \cdot \frac{|V|(|V| - 1)}{2}}{|V|(|V| - 1)} = 1 - 1 = 0.$$

Conversely, if $\mu(L) = 0$, then $\frac{2 |E|}{|V|(|V| - 1)} = 1$, which implies $|E| = \frac{|V|(|V| - 1)}{2}$,
This meaning that $E$ contains all unordered pairs of distinct vertices from $V$.
Hence, for every pair of distinct elements $x$ and $y$ in $V$, the subalgebra $\langle x, y\rangle$ is nilpotent.

(3) $\mu(L) = 1$ if and only if $|E|=0$, which holds if and only if $\langle x,y \rangle$ is non-nilpotent for all $ x,y\in V$.
\end{proof}

\begin{proposition}\label{invariants}
Let $L_{1}$ and $L_{2}$ be finite dimensional non-nilpotent Lie superalgebras.
If $L_{1}$ and $L_{2}$ are isomorphic as Lie superalgebras, then the following hold:
\begin{enumerate}[label=(\arabic*),font=\textnormal]
\item The nilpotent graphs $\mathfrak{G}(L_{1})$ and $\mathfrak{G}(L_{2})$ are isomorphic as graphs.

\item The non-nilpotent graphs $\mathfrak{G}(L_{1})^{c}$ and $\mathfrak{G}(L_{2})^{c}$ are isomorphic as graphs.
\end{enumerate}
\end{proposition}
\begin{proof}
(1) Denote the vertex set of $\mathfrak{G}(L_{i})$ by $V_{i} = L_{i} \setminus (\operatorname{nil}(L_{i})\cup \{0\})$, where $i=1,2$.
Let $\varphi: L_{1} \to L_{2}$ be an isomorphism of Lie superalgebras.
By Theorem \ref{12345}, we have

$$x \in \operatorname{nil}(L_{1}) \iff \varphi(x) \in \operatorname{nil}(L_{2}).$$
Therefore, the restriction of $\varphi$ to the set $L_1 \setminus (\operatorname{nil}(L_1)\cup \{0\})$ yields a bijection onto $L_2 \setminus (\operatorname{nil}(L_2)\cup \{0\})$.
Now, consider any two distinct vertices $x, y \in V_1$.
By definition, $x$ and $y$ adjacent in the nilpotent graph $\mathfrak{G}(L_1)$ if and only if the subalgebra $\langle x, y \rangle$ is nilpotent.
Since $\varphi$ is an isomorphism, then the subalgebra $\langle \varphi(x), \varphi(y) \rangle$ is nilpotent,
which in turn means that $\varphi(x)$ and $\varphi(y)$ are adjacent in $\mathfrak{G}(L_2)$.
Hence, $\varphi$ induces a graph isomorphism from $\mathfrak{G}(L_1)$ to $\mathfrak{G}(L_2)$.

(2) An entirely analogous argument applied to the definition of the non-nilpotent graph shows that
 $\varphi$ also induces a graph isomorphism from $\mathfrak{G}(L_1)^c$ to $\mathfrak{G}(L_2)^c$.
\end{proof}

\noindent \textbf{Remark.}
Proposition \ref{invariants} shows that the nilpotent graph and the non-nilpotent graph are isomorphic invariants of Lie superalgebras.
Conversely, the isomorphism of nilpotent graph and the non-nilpotent graph does not imply the isomorphism of Lie superalgebras.

In the following, we present a formula for the nilpotency measure of the direct sum of two Lie superalgebras.
Next, define $\binom{a}{b}$ as the binomial coefficient, and let $$A \times B=\{(\alpha,\beta)| \alpha\in A, \beta\in B\}.$$

\begin{definition}
Suppose that $L$ is a finite-dimensional non-nilpotent Lie superalgebra and $x,y\in L$. Define the map $I_{L}: L\times L\rightarrow \mathbb{Z}_{2}$ by
\[
I_{L}(x, y) = \begin{cases}
1, & \langle x, y \rangle \text{ is nilpotent}, \\
0, & \text{otherwise}.
\end{cases}
\]
\end{definition}

\begin{lemma}\label{leqwer}
Suppose that $L_1$ and $L_2$ are finite-dimensional non-nilpotent Lie superalgebras and $L=L_1\oplus L_2$. Then for any $u = (u_1, u_2), v = (v_1, v_2) \in L$,
we have
\begin{equation}\label{3.1}
I_{L}(u, v) = I_{L_1}(u_1, v_1) \cdot I_{L_2}(u_2, v_2).
\end{equation}
\end{lemma}
\begin{proof}
By Lemma \ref{lemma8}, $\langle u, v \rangle$
is nilpotent if and only if both $\langle u_1, v_1 \rangle$ and $\langle u_2, v_2 \rangle$ are nilpotent.
Hence equation \ref{3.1} is true.
\end{proof}

Let $L_1$ and $L_2$ be finite-dimensional non-nilpotent Lie superalgebras.
Set $A_i = L_i \setminus (\operatorname{nil}(L_i)\cup \{0\})$ and $a_i = |A_i|$, $B_i = (\operatorname{nil}(L_i)\cup \{0\})$ and $b_i = |B_i|$.
Define $\alpha_i = 1 - \mu(L_i)$, $\omega_{i}=\{ x \in A_i \mid \langle x \rangle \text{ is nilpotent} \}$
and $\nu_i = \dfrac{ |\omega_{i}| }{a_i}$, where $i=1,2.$

\begin{proposition}
Let $L_1$ and $L_2$ be finite-dimensional non-nilpotent Lie superalgebras with $\operatorname{char}\mathbb{F} \geq 3$.
Then the nilpotency measure of $L = L_1 \oplus L_2$ is
\[
\mu(L) = 1-\frac{2Q}{|V|(|V|-1)},
\]
where
\begin{align*}
|V| &= a_1 a_2 + a_1 b_2 + b_1 a_2, \\
Q &=Q_{11}+Q_{22}+Q_{33}+Q_{12}+Q_{13}+Q_{23}\\
&=\frac{1}{2} a_1 a_2 \Bigl( \alpha_1 \alpha_2 (a_1-1)(a_2-1) + \nu_1 \alpha_2 (a_2-1) + \alpha_1 \nu_2 (a_1-1) \Bigr) \\
+& \frac{1}{2}a_1 b_2 \bigl( \alpha_1 b_2 (a_1-1) + \nu_1 (b_2-1) \bigr) \\
+& \frac{1}{2}a_2 b_1 \bigl( \alpha_2 b_1 (a_2-1) + \nu_2 (b_1-1) \bigr) \\
+& a_1 a_2 b_2 \bigl( \alpha_1 (a_1-1) + \nu_1 \bigr) \\
+& a_1 a_2 b_1 \bigl( \alpha_2 (a_2-1) + \nu_2 \bigr) \\
+& a_1 a_2 b_1 b_2.
\end{align*}
In particular, if $B_i = \varnothing$, then
\[
\mu(L) = \frac{ 2- a_1 a_2 \Bigl( \alpha_1 \alpha_2 (a_1-1)(a_2-1) + \nu_1 \alpha_2 (a_2-1) + \alpha_1 \nu_2 (a_1-1) \Bigr) }{ 2a_1 a_2 (a_1 a_2 - 1) }.
\]
\end{proposition}

\begin{proof}
By Theorem \ref{theorem4}, we have $\operatorname{nil}(L) = \operatorname{nil}(L_1) \oplus \operatorname{nil}(L_2) = B_1 \times B_2$.
Thus the vertex set is
\[
V = L \setminus (\operatorname{nil}(L)\cup \{0\}) = (A_1 \times A_2) \cup (A_1 \times B_2) \cup (B_1 \times A_2),
\]
with size
\[
|V| = a_1 a_2 + a_1 b_2 + b_1 a_2.
\]

For any $u = (u_1, u_2), v = (v_1, v_2) \in V$,
we need to compute the number $Q$ of unordered pairs $\{u, v\}$ with $u \neq v$ in $V$ such that $I(u, v) = 1$.

Next, we consider three classes for $V$:
\begin{enumerate}[label=(\arabic*)]
    \item $C_1 = \{(\beta_{1},\beta_{2})| \beta_{1}\in A_1, \beta_{2}\in A_2\}$,
    \item $C_2 = \{(\beta_{1},\beta_{2})| \beta_{1}\in A_1, \beta_{2}\in B_2\}$,
    \item $C_3 = \{(\beta_{1},\beta_{2})| \beta_{1}\in B_1, \beta_{2}\in A_2\}$.
\end{enumerate}
For any $u \in C_i$ and $v \in C_j$ with $1 \leq i \leq j \leq 3$, compute the number of such vertex pairs with $I(u, v) = 1$.
By Lemma \ref{leqwer}, we need to compute  $I_1(u_1, v_1) = 1$ and $I_2(u_2, v_2) = 1$.
If $u_1 = v_1$ and $u_2 = v_2$, then $u = v$, which is not considered.

\noindent\textbf{Case 1:} \textbf{$u, v \in C_1$.}

Then $u_1, v_1 \in A_1$ and $u_2, v_2 \in A_2$.

\textbf{Subcase 1.1}. If $u_1 \neq v_1$ and $u_2 \neq v_2$, then the probability that $I_1(u_1, v_1) = 1$ is $\alpha_1$, and for $I_2(u_2, v_2) = 1$ is $\alpha_2$.  Hence, the number of this case is $\alpha_1 \alpha_2 \cdot \binom{a_1}{2} \binom{a_2}{2} \cdot 2$.

\textbf{Subcase 1.2}. If $u_1 = v_1$ but $u_2 \neq v_2$, then $I_1(u_1, v_1) = 1$ if and only if $\langle u_1 \rangle$ is nilpotent, with probability $\nu_1$; $I_2(u_2, v_2) = 1$ with probability $\alpha_2$. Hence, the number of this case is $\nu_1 \alpha_2 \cdot a_1 \binom{a_2}{2}$.

\textbf{Subcase 1.3}. If $u_1 \neq v_1$ and $u_2 = v_2$, similarly, the number of this case is $\alpha_1 \nu_2 \cdot \binom{a_1}{2} a_2$.

Thus, the total number of nilpotent pairs in this class is
$$
\begin{aligned}
Q_{11} &=\alpha_1 \alpha_2 \cdot \binom{a_1}{2} \binom{a_2}{2} \cdot 2 + \nu_1 \alpha_2 \cdot a_1 \binom{a_2}{2} + \alpha_1 \nu_2 \cdot \binom{a_1}{2} a_2 \\
       &=\frac{1}{2} a_1 a_2 \Bigr( \alpha_1 \alpha_2 (a_1-1)(a_2-1) + \nu_1 \alpha_2 (a_2-1) + \alpha_1 \nu_2 (a_1-1)\Bigr).
\end{aligned}
$$
\noindent\textbf{Case 2:} \textbf{$u, v \in C_2$.}

Here $u_1, v_1 \in A_1$ and $u_2, v_2 \in B_2$.
By the definition of $B_{2}$, we have $I_2(u_2, v_2) = 1$.
For the first component:

\textbf{Subcase 2.1}. If $u_1 \neq v_1$ and $u_2 \neq v_2$, then $I_1(u_1, v_1) = 1$ with probability $\alpha_1$. Hence, the number of this case is $\alpha_1 \cdot \binom{a_1}{2} \binom{b_2}{2} \cdot 2$.

\textbf{Subcase 2.2}. If $u_1 \neq v_1$ and $u_2 = v_2$, the number of this case is $\alpha_1 \cdot \binom{a_1}{2}b_{2}$.

\textbf{Subcase 2.3}. If $u_1 = v_1$ and $u_2 \neq v_2$, the number of this case is $\nu_1 \cdot a_1 \binom{b_2}{2}$.

Therefore,
$$
\begin{aligned}
Q_{22} &= \alpha_1 \cdot \binom{a_1}{2} [\binom{b_2}{2} \cdot 2 +b_{2}]+ \nu_1 \cdot a_1 \binom{b_2}{2}\\
       &=\frac{a_1 b_2}{2} \left( \alpha_1 b_2 (a_1-1) + \nu_1 (b_2-1) \right).
\end{aligned}
$$

\noindent\textbf{Case 3:} \textbf{$u, v \in C_3$.}

Similarly to Case 2,
$$
\begin{aligned}
Q_{33} &= \alpha_2 \cdot \binom{a_2}{2}[\binom{b_1}{2} \cdot 2+b_{1}] + \nu_2 \cdot \binom{b_1}{2} a_2\\
       &= \frac{a_2 b_1}{2} \left( \alpha_2 b_1 (a_2-1) + \nu_2 (b_1-1) \right).
\end{aligned}
$$

\noindent\textbf{Case 4:} \textbf{$u \in C_1$, $v \in C_2$.}

Here $u_1, v_1 \in A_1$, $u_2 \in A_2$, $v_2 \in B_2$.
Since $v_2 \in B_2$, we always have $I_2(u_2, v_2) = 1$.
For the first component:

\textbf{Subcase 4.1}. If $u_1 \neq v_1$, probability $\alpha_1$, the number of this case is $\alpha_1 \cdot (c_1 c_2 - a_1 a_2 b_2)$.

\textbf{Subcase 4.2}. If $u_1 = v_1$, probability $\nu_1$, the number of this case is $\nu_1 \cdot (a_1 a_2 b_2)$.

Note that $u_2 \neq v_2$ and $A_2 \cap B_2 = \varnothing$, so there is no equal case for the second component.
The total number of pairs is $c_1 \cdot c_2$, among which the number with $u_1 = v_1$ is $a_1 \cdot a_2 \cdot b_2$.
Thus,
$$
\begin{aligned}
Q_{12} &= \alpha_1 \cdot (c_1 c_2 - a_1 a_2 b_2) + \nu_1 \cdot (a_1 a_2 b_2)\\
       &= a_1 a_2 b_2 \bigl( \alpha_1 (a_1-1) + \nu_1 \bigr).
\end{aligned}
$$

\noindent\textbf{Case 5:} \textbf{$u \in C_1$, $v \in C_3$.}

Similarly,
$$
\begin{aligned}
Q_{13} &= \alpha_2 \cdot (c_1 c_3 - a_1 b_1 a_2) + \nu_2 \cdot (a_1 b_1 a_2)\\
       &= a_1 a_2 b_1 \bigl( \alpha_2 (a_2-1) + \nu_2 \bigr).
\end{aligned}
$$

\noindent\textbf{Case 6:} \textbf{$u \in C_2$, $v \in C_3$.}

Since $u_1 \in A_1$ and $v_1 \in B_1$, we have $I_1(u_1, v_1) = 1$.
Similarly, $I_2(u_2, v_2) = 1$.
Hence all such pairs are nilpotent,
$$
\begin{aligned}
Q_{23} = c_2 \cdot c_3=a_1 a_2 b_1 b_2.
\end{aligned}
$$

So
$$
\begin{aligned}
Q &=Q_{11}+Q_{22}+Q_{33}+Q_{12}+Q_{13}+Q_{23}\\
&=\frac{1}{2} a_1 a_2 \Bigl( \alpha_1 \alpha_2 (a_1-1)(a_2-1) + \nu_1 \alpha_2 (a_2-1) + \alpha_1 \nu_2 (a_1-1) \Bigr) \\
+& \frac{1}{2}a_1 b_2 \bigl( \alpha_1 b_2 (a_1-1) + \nu_1 (b_2-1) \bigr) \\
+& \frac{1}{2}a_2 b_1 \bigl( \alpha_2 b_1 (a_2-1) + \nu_2 (b_1-1) \bigr) \\
+& a_1 a_2 b_2 \bigl( \alpha_1 (a_1-1) + \nu_1 \bigr) \\
+& a_1 a_2 b_1 \bigl( \alpha_2 (a_2-1) + \nu_2 \bigr) \\
+& a_1 a_2 b_1 b_2.
\end{aligned}
$$

In particular, when $B_i = \varnothing$, only Case 1 exists. Hence
\[
\mu(L) = \frac{ 2- a_1 a_2 \Bigl( \alpha_1 \alpha_2 (a_1-1)(a_2-1) + \nu_1 \alpha_2 (a_2-1) + \alpha_1 \nu_2 (a_1-1) \Bigr) }{ 2a_1 a_2 (a_1 a_2 - 1) }.
\]
\end{proof}

After clarifying the definitions and basic properties of the nilpotency measure,
a natural question arises: Can the nilpotency measure approximate any arbitrarily specified value?
Based on the direct sum structure of Lie superalgebras and the calculation method of the nilpotency measure established earlier,
we thus obtain the following proposition. Below, $\mathbb{F}_{q}$ denotes an algebraically closed field of characteristic $q$.

\begin{theorem}
For every integer $q \geq 2$ and every $\varepsilon > 0$, there exists a Lie superalgebra $L$ satisfying
$$\left| \mu(L) - \frac{q}{q+1} \right| < \varepsilon.$$
\end{theorem}

\begin{proof}
To prove that the proposition, we construct a family of finite-dimensional Lie superalgebras $L_{q,m}$.
Let $L_{q}=(L_{q})_{\overline{0}}\oplus (L_{q})_{\overline{1}}$ be a Lie superalgebra over $\mathbb{F}_{q}$ with
$(L_{q})_{\overline{0}} = \operatorname{span}\{x, y\}$ and $(L_{q})_{\overline{1}} = \{0\}$, which satisfies $[x,y] = y$.
Obviously, $\operatorname{nil}(L_{q}) = \{0\}$ and the vertex set $V_{L_{q}} = L_{q} \setminus \{0\}$. Its nilpotency measure is:
$$\mu(L_{q}) = \frac{q-2}{q^2 - 2}.$$

Let $N_{m}$ be an $m$-dimensional abelian Lie superalgebra over $\mathbb{F}_q$. We construct the direct sum
$$L_{q,m} = N_{m} \oplus L_{q},$$
where the Lie bracket is defined in the direct sum form: $[(n_{1}, s_{1}), (n_{2}, s_{2})] = (0, [s_{1}, s_{2}]_{L_{q}})$.
It is straightforward to verify that this is a Lie superalgebra.
By Theorem \ref{theorem4}, the nilpotentizer is

$$\operatorname{nil}(L_{q,m}) = \operatorname{nil}(N_{m}) \oplus \operatorname{nil}(L_{q}) = N_{m} \oplus \{0\} \cong N_{m}.$$

Thus, the vertex set is:

$$V = L_{q,m} \setminus (\operatorname{nil}(L_{q,m})\cup \{0\}) = \{(n, s) \mid n \in N_{m},\ s \in L_{q},\ s \neq 0\},$$
with size:
$$|V| = |N_{m}| \cdot (|L_{q}| - 1) = q^{m} \cdot (q^{2} - 1).$$

Next, we determine the edge set. The following equivalences hold:
$$
\begin{aligned}
  &(n_{1}, s_{1}) \text{ and } (n_{2}, s_{2}) \text{ are adjacent},\\
  \Leftrightarrow &\langle(n_{1}, s_{1}), (n_{2}, s_{2})\rangle_{L_{q,m}} \text{ is nilpotent},\\
  \Leftrightarrow &\langle s_{1}, s_{2}\rangle_{L_{q}}  \text{ is nilpotent}, \\
  \Leftrightarrow &s_{1} \text{ and } s_{2} \text{ are linearly dependent}.\\
\end{aligned}
$$

Partition the set of non-zero vectors in $L_{q}$ by 1-dimensional subspaces.
There are exactly $(q+1)$ such subspaces, each containing $q-1$ non-zero vectors.
The vertices of $V$ belonging to the same 1-dimensional subspace form a clique.
Consequently, the nilpotent graph $\mathfrak{G}(L_{q,m})$ consists of $q+1$ disjoint cliques, each of size is $q^{m} (q-1)$.
The total number of edges is therefore

$$|E| = (q+1) \cdot \binom{q^{m} (q-1)}{2} = (q+1) \cdot \frac{q^{m} (q-1)(q^{m} (q-1)-1)}{2}.$$

Finally, we compute the nilpotency measure:

\begin{align*}
\mu(L_{q,m}) &= 1-\frac{2|E|}{|V|(|V|-1)} \\
&= 1-\frac{(q+1) q^{m} (q-1)(q^{m} (q-1)-1)}{(q+1)q^{m} (q-1) \cdot ((q+1)q^{m} (q-1) - 1)} \\
&= \frac{q^{m+2}-q^{m+1}}{q^{m+2}-q^{m}-1}\\
&= \frac{q}{q+1-\frac{1}{q^{m}}}.
\end{align*}

Obviously, $\lim\limits_{m\rightarrow \infty} \mu(L_{q,m})=\frac{q}{q+1}$. This completes the proof.
\end{proof}

\section{The category perspective of the nilpotentizer}
\ \ \ \
In this section, we describe the nilpotentizer $\operatorname{nil}(L)$ for a Lie superalgebra from the perspectives of category theory.

\begin{proposition}\label{pro1}
Let $\mathcal{L}$ denote the category defined as follows:
\begin{itemize}
    \item \textbf{Objects}: Finite dimensional Lie superalgebras;
    \item \textbf{Morphisms}: Surjective homomorphisms of Lie superalgebras.
\end{itemize}

Define a map $\mathcal{N}:\mathcal{L}\to\mathcal{L}$ by specifying its action:
\begin{itemize}
    \item \textbf{On objects}: For any finite dimensional Lie superalgebra $L\in\mathcal{L}$, define $\mathcal{N}(L)=\mathrm{nil}(L)$ ;
    \item \textbf{On morphisms}: For any morphism $f:L_{1}\to L_{2}$ in $\mathcal{L}$, let $\mathcal{N}(f)$ be the restriction of $f$ to $\mathrm{nil}(L_{1})$, that is,
    $\mathcal{N}(f) = f|_{\mathrm{nil}(L_{1})}:\mathrm{nil}(L_{1})\to\mathrm{nil}(L_{2})$.
\end{itemize}

Then $\mathcal{N}$ is a functor from $\mathcal{L}$ to itself.
\end{proposition}
\begin{proof}
Obviously, $\mathcal{L}$ is a category. By Theorem \ref{12345}(1), for any homomorphism $f: L_{1} \to L_{2}$, $f(\text{nil}(L_{1})) \subseteq \text{nil}(L_{2})$.
This ensures that the restriction $\mathcal{N}(f)$ is well-defined.
To verify that $\mathcal{N}$ is a functor, we need to check two key functorial properties:
the preservation of identity morphisms and the preservation of composition of morphisms.

(1) Preservation of Identity Morphisms.

Let $L$ be an object in $\mathcal{L}$, and let $\operatorname{id}_{L}: L \to L$ be its identity morphism. By the definition of the functor $\mathcal{N}$,
we have $\mathcal{N}(L) = \operatorname{nil}(L)$ on objects, and $\mathcal{N}(\operatorname{id}_{L}) = \operatorname{id}_{L}|_{\operatorname{nil}(L)}$ on morphisms.
We now show that this restricted map is the identity morphism on $\mathcal{N}(L)$. For any $x \in \operatorname{nil}(L)$, it follows directly that
$$\mathcal{N}(\operatorname{id}_{L})(x) = \operatorname{id}_{L}|_{\operatorname{nil}(L)}(x) = \operatorname{id}_{L}(x) = x.$$

Therefore, $\mathcal{N}(\operatorname{id}_{L}) = \operatorname{id}_{\operatorname{nil}(L)} = \operatorname{id}_{\mathcal{N}(L)}$,
which proves that $\mathcal{N}$ preserves identity morphisms.

(2) Preservation of Composition of Morphisms.

Let $f: L \to M$ and $g: M \to N$ be morphisms in $\mathcal{L}$. We show that $\mathcal{N}(g \circ f)=\mathcal{N}(g) \circ \mathcal{N}(f)$.
For any $x \in \text{nil}(L)$, we compute both sides:
\begin{equation*}
\mathcal{N}(g \circ f)(x)=(g \circ f)|_{\text{nil}(L)}(x)=(g \circ f)(x)=g(f(x)).
\end{equation*}
On the other hand,
\begin{equation*}
\mathcal{N}(g) \circ \mathcal{N}(f)(x)=g|_{\text{nil}(M)}(f|_{\text{nil}(L)}(x))=g(f(x)).
\end{equation*}

We conclude that $\mathcal{N}(g \circ f)(x)=\mathcal{N}(g) \circ \mathcal{N}(f)(x)$ for all $x \in \text{nil}(L)$,
which implies $\mathcal{N}(g \circ f)=\mathcal{N}(g) \circ \mathcal{N}(f)$.
Since $\mathcal{N}$ preserves both identity morphism and composition, it is a functor from $\mathcal{L}$ to $\mathcal{L}$.
\end{proof}

\begin{theorem}
Suppose that $A,\ B,\ C$ are Lie superalgebras and let $$0 \to A \xrightarrow{\alpha} B \xrightarrow{\beta} C \to 0$$ be a short exact sequence.
There exist natural homomorphisms:
\begin{equation*}
\alpha': \mathrm{nil}(A) \to B, \quad \beta': \mathrm{nil}(B) \to \mathrm{nil}(C)
\end{equation*}
such that $\ker \beta' \subseteq \alpha'(\mathrm{nil}(A))$.
\end{theorem}
\begin{proof}
Let $\alpha' = \alpha|_{\operatorname{nil}(A)}$ and $\beta' = \beta|_{\operatorname{nil}(B)}$.
First, we show that $\beta'$ maps into $\operatorname{nil}(C)$.
Let $x \in \operatorname{nil}(B)$ and $c \in C$.
Due to the surjectivity of $\beta$, there exists $b \in B$ such that $\beta(b) = c$.
Moreover, since the subalgebra $\langle x, b \rangle$ is nilpotent and $\beta$ is a homomorphism,
the image $\beta(\langle x, b \rangle) = \langle \beta(x), \beta(b) \rangle$ is also nilpotent.
Hence $\beta(x) \in \operatorname{nil}(C)$, which confirms that $\beta'$ is well-defined.
Next, we prove that $\ker(\beta') \subseteq \alpha'(\operatorname{nil}(A))$.
Let $x \in \ker \beta' \subseteq \operatorname{nil}(B)$. Then $\beta(x) = 0$.
Hence $x \in \ker \beta = \operatorname{im} \alpha$. There exists $a \in A$ such that $\alpha(a) = x$.
We now prove that $a \in \operatorname{nil}(A)$.

For any $a' \in A$, since $\alpha$ is an injective homomorphism, we have
$$
\alpha(\langle a, a' \rangle) = \langle \alpha(a), \alpha(a') \rangle = \langle x, \alpha(a') \rangle.
$$
The subalgebra $\langle x, \alpha(a') \rangle$ is nilpotent because $x \in \operatorname{nil}(B)$.
Since $\alpha$ is an isomorphism onto its image, the subalgebra $\langle a, a' \rangle$ must also be nilpotent.
Therefore, $a \in \operatorname{nil}(A)$, completing the proof.
\end{proof}

\begin{theorem}
Suppose that $f: L \to M$ and $g: N \to M$  are morphisms in the category $\mathcal{L}$.
Define
$$P =L \times_M N=\{(x, y) \in L \oplus N \mid f(x) = g(y)\},$$
and let $p_L: P \to L$ and $p_N: P \to N$ be the canonical projections.
Then the following statements hold:
\begin{enumerate}[label=(\arabic*),font=\textnormal]
\item The pair $(p_L,p_N)$ is the pullback of $f$ and $g$.

\item $\mathcal{N}(P) =\mathcal{N}(L) \times_{\mathcal{N}(M)} \mathcal{N}(N)$,
where
$$\mathcal{N}(L) \times_{\mathcal{N}(M)} \mathcal{N}(N)=\{(x, y) \in \mathcal{N}(L) \oplus \mathcal{N}(N) \mid \mathcal{N}(f)(x) = \mathcal{N}(g)(y)\}.$$
\end{enumerate}
\end{theorem}

\begin{proof}
(1) It is clear that $P$ is a subalgebra of $L \oplus N$.
Take any $x \in L$. Since $g$ is surjective, there exists $y \in N$ such that $g(y) = f(x)$.
Then $(x, y) \in P$ and $p_{L}(x, y) = x$, So $p_L$ is surjective. Similarly, $p_N$ is surjective.

Let $K$ be an object in $\mathcal{L}$, and let $u: K \to L$ and $v: K \to N$ be epimorphisms satisfying $f \circ u = g \circ v$.
Define a map $h:K\rightarrow P$ by $h(k) = (u(k), v(k))$ for all $k \in K$.
The condition $f \circ u = g \circ v$ ensures that $f(u(k)) = g(v(k))$, so $h(k) \in P$.
By construction, $h$ is a morphism in $\mathcal{L}$.
For the uniqueness of $h$, suppose $h': K \to P$ is another morphism such that $p_L \circ h' = u$ and $p_N \circ h' = v$.
Then for any $k \in K$,
$$h'(k) = (p_L(h'(k)), p_N(h'(k))) = (u(k), v(k)) = h(k).$$
Therefore, $h'=h$.

(2) Suppose that $(x, y) \in \mathcal{N}(P)$. Then for any $(a, b) \in P$, the subalgebra $\langle (x, y), (a, b) \rangle$ is nilpotent.
By Lemma \ref{lemma8}, This is equivalent to $\langle x, a \rangle$ and $\langle y, b \rangle$ being nilpotent.
For any $a \in L$. Since $g$ is surjective, there exists $b \in N$ such that $g(b) = f(a)$. Hence $(a, b) \in P$ and $\langle x, a \rangle$ is nilpotent,
this proves $x \in \mathcal{N}(L)$. Similarly, $y \in \mathcal{N}(N)$. Moreover, the condition $f(x) = g(y)$ holds,
it follows that
$$\mathcal{N}(P) \subseteq\mathcal{N}(L) \times_{\mathcal{N}(M)} \mathcal{N}(N).$$
Conversely, suppose that $(x, y) \in \mathcal{N}(L) \times_{\mathcal{N}(M)} \mathcal{N}(N)$.
Then for any $(a, b) \in P$, the subalgebras $\langle x, a \rangle$ and $\langle y, b \rangle$ are nilpotent.
Applying Lemma \ref{lemma8} again, we conclude that $\langle (x, y), (a, b) \rangle$ is nilpotent, which means $(x, y) \in \mathcal{N}(P)$.
Thus, $$\mathcal{N}(L) \times_{\mathcal{N}(M)} \mathcal{N}(N)\subseteq\mathcal{N}(P).$$
Therefore, $$\mathcal{N}(P) = \mathcal{N}(L) \times_{\mathcal{N}(M)} \mathcal{N}(N).$$
\end{proof}

\section{Acknowledgement}
The authors thank Professor Dimitry Leites for his comments on the original version of the paper and Professor Ke Ou for the useful discussions.

\end{document}